\documentclass[reqno]{amsart}

\usepackage[utf8]{inputenc}
\usepackage{amsmath}
\usepackage{amsthm}
\usepackage{amsfonts}
\usepackage{amssymb}
\usepackage{mathrsfs}
\usepackage{dsfont}
\usepackage[]{hyperref}
\usepackage{cite}
\usepackage{graphicx}
\usepackage{pgfplots}
\pgfplotsset{compat = newest}

\newtheorem{theorem}{Theorem}
\theoremstyle{plain}

\newtheorem{corollary}{Corollary}
\newtheorem{definition}{Definition}

\newtheorem{lemma}{Lemma}

\newtheorem{remark}{Remark}
\numberwithin{equation}{section}
\numberwithin{theorem}{section}
\numberwithin{lemma}{section}
\numberwithin{definition}{section}
\numberwithin{corollary}{section}
\numberwithin{remark}{section}

\newcommand{\naturals}{\mathbb{N}}

\newcommand{\complexes}{\mathbb{C}}

\newcommand{\spec}[1]{\operatorname{\sigma}{\left(#1\right)}}
\newcommand{\proj}[1]{\mathsf{P}_{#1}}

\newcommand{\id}[1]{\mathrm{id}_{#1}}

\newcommand{\hadj}{*}
\newcommand{\iprod}[2]{\left\langle #1, #2 \right\rangle}

\newcommand{\cbnorm}[1]{\left\| #1 \right\|_{\mathrm{cb}}}

\newcommand{\hilbert}{\mathcal{H}}

\newcommand{\matr}[1]{\mathbb{M}_{#1}(\complexes)}
\newcommand{\matra}[2]{\mathbb{M}_{#1}(#2)}

\newcommand{\unit}{\mathds{1}}

\newcommand{\cp}[2]{\mathsf{CP}(#1,#2)}
\newcommand{\cpe}[1]{\mathsf{CP}(#1)}

\newcommand{\gdec}[3]{\mathsf{D}_{#1}(#2,#3)}

\newcommand{\transpose}{t}

\begin{document}
\title[On a generalization of decomposable maps on C*-algebras]{On a generalization of decomposable maps on C*-algebras}
\author{Krzysztof Szczygielski}
\address{Institute of Theoretical Physics and Astrophysics, Faculty of Mathematics, Physics and Informatics, University of Gda\'nsk, Wita Stwosza 57, 80-308 Gda\'nsk, Poland}
\email{krzysztof.szczygielski@ug.edu.pl}

\begin{abstract}
We propose the notion of countable decomposability of maps on C*-algebras: a bounded linear map $\varphi : \mathscr{A}\to B(\hilbert)$, where $\mathscr{A}$ is a C*-algebra and $\hilbert$ a Hilbert space, will be called \emph{countably decomposable} if it admits a representation $\varphi = \sum_{k=1}^{\infty} \psi_k \circ \phi_k$ for completely positive maps $\psi_k : \mathscr{A}\to B(\hilbert)$ and bounded *-maps $\phi_k : \mathscr{A}\to\mathscr{A}$. A characterization of countable decomposability is given in certain cases with various assumptions imposed on maps $\phi_k$. Our findings provide extensions of a classical result of Størmer from Proc.~Amer.~Math.~Soc.~\textbf{86} (1982), 402-404, originally formulated for decomposable positive maps.
\end{abstract}

\maketitle

\section{Introduction}

Let $\mathscr{A}$, $\mathscr{B}$ be C*-algebras and $\hilbert$ a Hilbert space. A linear map $\varphi : \mathscr{A}\to \mathscr{B}$ is called a \emph{*-map} if it preserves self-adjoint elements, i.e.~when $\varphi(a^\hadj) = \varphi(a)^\hadj$. A bounded *-map is further called \emph{positive} if $\varphi (\mathscr{A}^+) \subseteq\mathscr{B}^+$ (with $\mathscr{M}^+$ being a convex cone of positive elements in $\mathscr{M}$). Moreover, $\varphi$ is called \emph{completely positive} (abbreviated c.p.) or \emph{completely copositive} (c.cp.), respectively, when $\id{n}\otimes\varphi$ or $\transpose_n \otimes \varphi$, resp.,~is positive as a map from $\matra{n}{\mathscr{A}}$ to $\matra{n}{\mathscr{B}}$, for all $n \in \naturals$ ($\id{n}$ and $\transpose_n$ are the identity map and transposition on $\matr{n}$, resp.). The set of all c.p.~maps from $\mathscr{A}$ to $\mathscr{B}$ is a convex cone in $B(\mathscr{A},\mathscr{B})$. Finally, $\varphi$ is called \emph{decomposable} if it can be expressed as
\begin{equation}
    \varphi = \phi_1 + \phi_2
\end{equation}
for a c.p.~map $\phi_1$ and a c.cp.~map $\phi_2$. Decomposable maps constitute for a nontrivial, yet mathematically manageable subclass of positive maps between C*-algebras. Their significance is especially noticeable in case of low-dimensional matrix algebras; for instance, it is known that all positive maps on $\matr{2}$, or between $\matr{2}$ and $\matr{3}$ (and vice versa), are in fact decomposable \cite{Woronowicz1976}. This statement becomes false in case of higher-dimensional algebras where nondecomposable positive maps were constructed (see e.g.~\cite{Choi1980,Miller2015,Rutkowski2015,Chruscinski2018}), already in the $\matr{3}$ case. Let us set $\mathscr{B} = B(\hilbert)$ for some Hilbert space $\hilbert$. Decomposable maps acting from $\mathscr{A}$ to $B(\hilbert)$ are characterized by the following dilation theorem:

\begin{theorem}[cf.~\cite{Stoermer2013}]
    A map $\varphi : \mathscr{A}\to B(\hilbert)$ is decomposable if and only if there exists a Hilbert space $\mathcal{K}$, a linear bounded operator $V : \hilbert\to\mathcal{K}$ and a Jordan *-morphism $\pi : \mathscr{A} \to B(\mathcal{K})$ such that $\varphi(a) = V^\hadj \pi (a) V$ for any $a\in\mathscr{A}$.
\end{theorem}

Then, the following classical result \cite{Stormer_1982} allows to further characterize decomposable maps by the way their extensions of a form $\varphi_n = \id{n}\otimes\varphi$, or \emph{ampliations}, behave on the tensor product algebras $\matra{n}{\mathscr{A}}$, $n\in\naturals$:

\begin{theorem}[cf.~\cite{Stormer_1982}]\label{thm:StormerDecEquiv}
    A map $\varphi : \mathscr{A}\to B(\hilbert)$ is decomposable if and only if, for all $\left[a_{ij}\right]_{i,j}\in \matra{n}{\mathscr{A}}$ such that both $\left[a_{ij}\right]_{i,j}$, $\left[a_{ji}\right]_{i,j}\in\matra{n}{\mathscr{A}}^+$, we have $\left[\varphi(a_{ij})\right]_{i,j}\in\matra{n}{B(\hilbert)}^+$, for all $n\in\naturals$.
\end{theorem}

The above result may be considered the main inspiration standing behind this article. The ordinary decomposability of positive maps was already demonstrated to be of a high mathematical value, either purely structurally or because of their usefulness in quantum physics and quantum information theory. Therefore it seems reasonable to seek for various possible extensions of decomposability, especially since they may potentially provide a better understanding of the structure of cone of positive maps itself. Our aim is to define a particular generalization of decomposable maps from $\mathscr{A}$ to $B(\hilbert)$, which we call \emph{countably decomposable}, by extending a certain algebraic construction invented in the proof of Theorem \ref{thm:StormerDecEquiv} in \cite{Stormer_1982}. This is achieved in Section \ref{sec:GDM} by two means: first, we allow an arbitrary number of terms in the decomposition, and second, we do not \emph{a priori} restrict ourselves to positive maps. Main results of the paper are presented in the form of Theorems \ref{thm:FinDecEquivCond}, \ref{thm:GenDecEquivCond} and \ref{thm:GenDecEquivCondNonunital} which provide a characterization of countable decomposability, for both unital and nonunital C*-algebras and under various technical conditions, and extend the scope of Theorem \ref{thm:StormerDecEquiv}. Some further properties, such as a convex cone structure and closedness in the BW topology, are also explored.

\section{Countable decomposability}
\label{sec:GDM}

We start with a formal definition of countable decomposability. We will write $\cp{\mathscr{A}}{\hilbert}$ for the convex cone of c.p.~maps from $\mathscr{A}$ to $B(\hilbert)$. For the following, let $\tau$ be some vector topology on $B(\hilbert)$.

\begin{definition}\label{def:GDM}
    Let $\mathscr{A}$ be a C*-algebra, $\hilbert$ a Hilbert space and $\phi = \left(\phi_k\right)_k$ a sequence of *-maps on $\mathscr{A}$. A map $\varphi \in B(\mathscr{A},B(\hilbert))$ will be called countably decomposable with respect to $\phi$, or $\phi$-decomposable for short, if there exists a sequence $\left(\psi_k\right)_k \in \cp{\mathscr{A}}{\hilbert}^\naturals$ such that $\varphi$ may be expressed as a series
    \begin{equation}\label{eq:GenDecDefinition}
        \varphi = \sum_{k=1}^{\infty} \psi_k \circ \phi_k
    \end{equation}
    converging pointwise in $\tau$ in $B(\mathscr{A},B(\hilbert))$. Set of all maps of a form \eqref{eq:GenDecDefinition} will be denoted $\gdec{\phi}{\mathscr{A}}{\hilbert}$.
\end{definition}
The definition of countable decomposability automatically forces the following necessary condition:
\begin{theorem}
    If $\varphi \in B(\mathscr{A},B(\hilbert))$ is $\phi$-decomposable then necessarily
    \begin{equation}\label{eq:KernelCondition}
        \bigcap_{k\in\naturals}\ker{\phi_k} \subseteq \ker{\varphi}.
    \end{equation}
\end{theorem}

\begin{proof}
    Taking $a\in\ker{\phi_k}$ for all $k$ yields the sequence of partial sums $\sigma_n = \sum_{k=1}^{n} \psi_k \circ \phi_k (a) = 0$, for all $n$. Then $\sigma_n \xrightarrow{\tau} \varphi(a)$ in $B(\hilbert)$ yields $\varphi(a) = 0$.
\end{proof}

Later on we will elaborate on the case when $\left(\phi_k\right)_k$ consists only of a finite number of nonzero elements. Naturally, maps $\varphi$ expressible in such manner will be called \emph{finitely decomposable}.

\begin{remark}
    As a rather immediate consequence of the very definition, we see that all positive decomposable maps are also finitely decomposable. Indeed, any decomposable map may be put in a form $\varphi = \psi_1 + \psi_2 \circ \transpose$, where $\psi_1, \psi_2 \in \cp{\mathscr{A}}{\hilbert}$ and $t$ is a transposition on $\mathscr{A}$ understood in the sense of a faithful representation of $\mathscr{A}$ on some Hilbert space. Therefore, every decomposable map is $(\id{},\transpose)$-decomposable; trivially, the same observation applies to c.p.~and c.cp.~maps.
\end{remark}

\subsection{General properties}

\subsubsection{Nonuniqueness}

It is evident that the decomposition \eqref{eq:GenDecDefinition} is non-unique: in principle, many different sequences $\left(\psi_k\right)_k$, $\left(\phi_k\right)_k$ may serve for representing the same map $\varphi$. For example, any permutation of a sequence $(\id{},\transpose,0,0,...)$ would still generate the same class of decomposable maps. Similar can be then said about any other sequence $\left(\phi_k\right)_k$ of decomposable maps since compositions $\psi_k\circ\phi_k$ are still decomposable by the mapping cone property of both c.p.~and c.cp.~maps. This freedom of choice introduces a lot of redundancy from a point of view of generation of countable decomposability. For some set $X$ of bounded *-maps on $\mathscr{A}$ let us introduce an equivalence relation $\sim$ between sequences
\begin{equation}
    \phi \sim \phi^\prime \,\,\,\Leftrightarrow \,\,\, \gdec{\phi}{\mathscr{A}}{\hilbert} = \gdec{\phi^\prime}{\mathscr{A}}{\hilbert}.
\end{equation}
Then, an equivalence class $\left[\phi\right]_\sim \in X^\naturals / \sim$ generates a set of $\phi$-decomposable maps, $\gdec{\left[\phi\right]_\sim}{\mathscr{A}}{\hilbert} := \gdec{\phi}{\mathscr{A}}{\hilbert}$ for any $\phi\in X^\naturals$. It is then straightforward to see the following:

\begin{corollary}
    Cones of c.p., c.cp.~and decomposable maps in $B(\mathscr{A},B(\hilbert))$ are special cases of set $\gdec{[\phi]_\sim}{\mathscr{A}}{\hilbert}$, respectively generated by equivalence classes
    \begin{equation}
        \left[(\id{}, 0, ...)\right]_\sim, \quad \left[(\transpose, 0, ...)\right]_\sim \quad \text{and} \quad \left[(\id{}, \transpose, 0, ...)\right]_\sim.
    \end{equation}
\end{corollary}

Evidently, it therefore should be of interest to seek for possibly the simplest (e.g.~shortest) representative of a given class $\left[\phi\right]_\sim$ to describe $\phi$-decomposable maps without loosing any generality.

\subsubsection{Cone structure}

It comes with no surprise that the set $\gdec{\phi}{\mathscr{A}}{\hilbert}$ of countably decomposable maps is, similarly to ordinary decomposable maps, a convex cone:

\begin{theorem}\label{thm:DphiConvCone}
    Set $\gdec{\phi}{\mathscr{A}}{\hilbert}$ is a convex cone in $B(\mathscr{A},B(\hilbert))$.
\end{theorem}

\begin{proof}
    Take any $\varphi_1, \varphi_2 \in \gdec{\phi}{\mathscr{A}}{\hilbert}$ so that sequences of partial sums
    \begin{equation}
        \sigma_{n}^{(i)}(a) = \sum_{k=1}^{n}\psi^{(i)}_{k}\circ\phi_k (a)
    \end{equation}
    with $\psi^{(i)}_k$ c.p., converge to $\varphi_i (a)$ in $\tau$, for all $a\in\mathscr{A}$. Let $\tilde{\varphi}_n (a) = \lambda_1 \sigma_{n}^{(1)}(a) + \lambda_2 \sigma_{n}^{(2)}(a)$, where $\lambda_1 , \lambda_2 > 0$. Since $\tau$ is a vector topology, we have
    \begin{equation}
        \lambda_i \sigma_{n}^{(i)}(a) \xrightarrow{\tau} \lambda_i \varphi_i (a), \quad \tilde{\varphi}_n (a) \xrightarrow{\tau} \lambda_1 \varphi_1 (a) + \lambda_2 \varphi_2 (a) ,
    \end{equation}
    for all $a$, by continuity of linear space operations under $\tau$. Therefore
    \begin{equation}\label{eq:ConvexSum}
        \sum_{k=1}^{\infty} \left( \lambda_1\psi^{(1)}_k+\lambda_2\psi^{(2)}_k \right)\circ\phi_k = \lambda_1 \varphi_1 + \lambda_2 \varphi_{2}
    \end{equation}
    converging pointwise in $\tau$. Since $\lambda_1\psi^{(1)}_k+\lambda_2\psi^{(2)}_k$ is c.p., \eqref{eq:ConvexSum} is $\phi$-decomposable. Hence $\gdec{\phi}{\mathscr{A}}{\hilbert}$ is closed with respect to conical combinations of its elements, and thus is a convex cone.
\end{proof}

\section{Special cases}

In this section we will focus our attention on some particular ``types'' of countable decomposability, specified by different choices for the defining sequence $\phi$ and its properties. We make two restrictions in our considerations, namely
\begin{itemize}
    \item we assume $\mathscr{A}$ has a unit $\unit_\mathscr{A}$ unless stated otherwise, and
    \item we restrict to $\tau$ being the \emph{norm topology in $B(\hilbert)$}.
\end{itemize}

 It would be interesting to examine some other, weaker choices for $\tau$; we mark this as an interesting problem for further research. That said, from here onwards we assume the decomposition \eqref{eq:GenDecDefinition} to converge pointwise in norm. The characterization of maps constructed in each case will make use of a family of convex cones
\begin{equation}\label{eq:GammaCones}
    \Gamma_{n}^{+} (\phi) = \left\{ \left[a_{ij}\right]_{i,j} \in \matra{n}{\mathscr{A}} : \left[\phi_k (a_{ij})\right]_{i,j} \in \matra{n}{\mathscr{A}}^+  \text{ for all } k\right\}
\end{equation}
defined for a given sequence $\phi$ and $n\in\naturals$. We remark here that the definition of $\Gamma_{n}^{+} (\phi)$ is clearly influenced by the condition appearing in the original Theorem \ref{thm:StormerDecEquiv}: indeed, positivity of both $\left[a_{ij}\right]_{i,j}$ and $\left[a_{ji}\right]_{i,j}$ is equivalent to $\left[a_{ij}\right]_{i,j}$ being a member of $\Gamma_{n}^{+} ((\id{},\transpose))$.

\subsection{Finitely decomposable maps}

For the start, we will address the case when $\phi$ contains only a finite number of nonzero elements, i.e.~when it is effectively an $m$-tuple; this case is the easiest to deal with and offers possibly the most regular structure.

\begin{definition}\label{def:FDM}
    Let $\mathscr{A}$ be a C*-algebra, $\hilbert$ a Hilbert space and $\phi = (\phi_k)_{k=1}^{m}$ an ordered $m$-tuple of bounded *-maps on $\mathscr{A}$. A $\phi$-decomposable map $\varphi \in B(\mathscr{A},B(\hilbert))$, i.e.~admitting a form
    \begin{equation}\label{eq:FinDecDefinition}
        \varphi = \sum_{k=1}^{m} \psi_k \circ \phi_k
    \end{equation}
    for an ordered $m$-tuple $(\psi_k)_{k=1}^{m} \in \cp{\mathscr{A}}{\hilbert}^m$, will be called finitely decomposable.
\end{definition}

\subsubsection{Characterization}

The following dilation theorem allows to express the finite decomposability in the spirit of Stinespring and Størmer:

\begin{theorem}\label{thm:FinDecDilation}
    If a map $\varphi \in B(\mathscr{A},B(\hilbert))$ is $\phi$-decomposable then there exists a Hilbert space $\mathcal{K}$, a bounded operator $V : \hilbert \to \mathcal{K}$ and a *-map $\rho : \mathscr{A} \to B(\mathcal{K})$ such that
    \begin{equation}
        \varphi (a) = V^\hadj \rho(a) V
    \end{equation}
    for all $a \in \mathscr{A}$. 
\end{theorem}

\begin{proof}
Assume $\varphi$ admits a decomposition \eqref{eq:FinDecDefinition}. By Stinespring's dilation theorem \cite{Stinespring_1955} there exist Hilbert spaces $\mathcal{K}_k$, bounded linear operators $V_k : \hilbert \to \mathcal{K}_k$ and *-homomorphisms $\pi_k : \mathscr{A} \to B(\mathcal{K}_k)$ such that
\begin{equation}
    \psi_k (a) = V_{k}^{\hadj} \pi_k (a) V_k
\end{equation}
for all $a\in\mathscr{A}$, $k \in \{1, ... , m\}$. Define $\mathcal{K} = \bigoplus_{k=1}^{m} \mathcal{K}_k$, together with orthogonal projections $\proj{k} : \mathcal{K}\to \mathcal{K}_k$. Set a bounded operator $V : \hilbert \to \mathcal{K}$ by
\begin{equation}
    Vh = \sum_{k=1}^{m} V_k h, \quad h \in \hilbert,
\end{equation}
and $\rho : \mathscr{A} \to B(\mathcal{K})$ by
\begin{equation}\label{eq:RhoGenForm}
    \rho (a) x = \sum_{k=1}^{m} \proj{k} \pi_k \circ \phi_k (a) \proj{k} x, \quad x \in \mathcal{K}.
\end{equation}
Then the action of $V^\hadj$ on $x\in\mathcal{K}$ can be checked to be $V^\hadj x = \sum_{k=1}^{m} V_{k}^{\hadj}\proj{k}x$. Since all $\phi_k$ are unital *-maps and $\pi_k$ are *-homomorphisms one quickly checks $\rho(a^\hadj) = \rho (a)^\hadj$. Finally, with a straightforward computation involving mutual orthogonality of $\proj{k}$ we have, for all $a\in\mathscr{A}$, $h \in \hilbert$,
\begin{equation}
    V^\hadj \rho (a) Vh = \sum_{k=1}^{m} V_{k}^{\hadj} \pi_k \circ \phi_k (a) V_k h = \sum_{k=1}^{m} \psi_k \circ \phi_k (a) h,
\end{equation}
i.e.~we recover $\varphi(a)h$.
\end{proof}
Conversely, if $\rho : \mathscr{A} \to B(\mathcal{K})$ admits a form
\begin{equation}
    \rho(a) = \sum_{k=1}^{m} \pi_k \circ \phi_k (a)
\end{equation}
for *-homomorphisms $\pi_k : \mathscr{A}\to B(\mathcal{K}_k)$ and *-maps $\phi_k \in B(\mathscr{A})$ then the map $a\mapsto V^\hadj \rho (a) V$ is of a form \eqref{eq:FinDecDefinition}.

As we show in Theorem \ref{thm:FinDecEquivCond} below, restricting to \emph{unital} maps $\phi_k$ allows for a full characterization of finite decomposability. This is the most direct expansion of the result of \cite{Stormer_1982}.

\begin{theorem}\label{thm:FinDecEquivCond}
    Let $\mathscr{A}$ be a unital C*-algebra and let $\left(\phi_k\right)_{k=1}^{m}$ be an ordered $m$-tuple of bounded unital *-maps on $\mathscr{A}$. A map $\varphi : \mathscr{A}\to B(\hilbert)$ is finitely decomposable with respect to $\phi$ if and only if
    \begin{equation}\label{eq:OnlyIfAssum}
        \forall \, n\in\naturals : \left[a_{ij}\right]_{i,j} \in \Gamma_{n}^{+}(\phi) \Rightarrow \left[\varphi(a_{ij})\right]_{i,j} \in \matra{n}{B(\hilbert)}^+ .
    \end{equation}
\end{theorem}

Before we proceed with a proof, we will first show some secondary results. For the following considerations we implicitly assume the ``only if'' part of the Theorem \ref{thm:FinDecEquivCond} holds, i.e.~the condition \eqref{eq:OnlyIfAssum}. Let us define

\begin{equation}\label{eq:WdefinitionCstar}
    \mathscr{W} = \matra{n}{\mathscr{A}} \simeq \matr{n} \otimes \mathscr{A}
\end{equation}
which is also a unital C*-algebra with a unit $\unit_\mathscr{W} = I_n \otimes \unit_\mathscr{A}$ ($I_n$ denotes the identity matrix). Let $E_{kl}$ be matrix units spanning $\matr{n}$ and let $\mathscr{D} \subset \mathscr{W}$ be a subspace of all block-diagonal matrices. Consider a function $\xi : \mathscr{A}\to\mathscr{D}$ defined by the equality
\begin{equation}\label{eq:xiDef}
    \xi(a) = \left(\begin{array}{cccc}
        \phi_1(a) & 0 & \cdots & 0\\
        0 & \phi_2(a) & \cdots & 0\\
        \vdots & \vdots & \ddots & \vdots \\
        0 & 0 & \cdots & \phi_m(a)
    \end{array}
\right) = \sum_{k=1}^{m} E_{kk}\otimes \phi_k(a).
\end{equation}
and denote $\mathscr{X} = \xi (\mathscr{A})$. One checks that since all $\phi_k$ are linear unital *-maps, $\mathscr{X}$ is an \emph{operator system} in $\mathscr{W}$ i.e.~a *-invariant subspace containing $\unit_\mathscr{W}$. Let $\Phi : \mathscr{X} \to B(\hilbert)$ be defined by the expression
\begin{align}\label{eq:PhiDef}
    \varphi(a) = \left(\Phi\circ\xi\right)(a).
\end{align}
We first make sure $\Phi$ is a well-defined linear map. Since all maps $\varphi$, $\phi_k$ are linear and $\mathscr{A}$ is a Banach space, equalities $\Phi(x+y) = \Phi(x) + \Phi(y)$ and $\Phi(\lambda x) = \lambda\Phi(x)$ naturally hold for all $x,y\in\mathscr{X}$, $\lambda\in\complexes$. Choose arbitrary $n \geqslant 2$ and take any $c\in\ker{\phi_k}$ for all $k$. Consider a diagonal matrix of size $n$,
\begin{equation}
    \left[a_{ij}\right]_{i,j} = \operatorname{diag}{\left\{ c,-c,0,...,0 \right\}}.
\end{equation}
Then $\left[\phi_k (a_{ij})\right]_{i,j} = 0$ which is a positive element in $\matra{n}{\mathscr{A}}$, for any $k$. By the assumption \eqref{eq:OnlyIfAssum},
\begin{equation}
    \left[\varphi(a_{ij})\right]_{i,j} = \operatorname{diag}{\left\{\varphi (c),-\varphi(c),0,...,0\right\}} \in \matra{n}{B(\hilbert)}^+ ,
\end{equation}
i.e.~$\varphi(c)$ and $-\varphi(c)$ are both positive. This is however possible if and only if $\varphi(c) = 0$; hence $c\in\ker{\varphi}$ and the necessary consistency condition \eqref{eq:KernelCondition} is automatically satisfied. Element of $\mathscr{X}$ is a zero vector iff it is of a form $\xi(c)$ for some $c\in\bigcap_{k\in\naturals}\ker{\phi_k}$. For such $c$ we therefore have $0 = \varphi(c) = \Phi(0)$. It remains to show $x\mapsto\Phi(x)$ is well defined as a function, i.e.~$x$ determines $\Phi(x)$ uniquely. For this, take any $x\in\mathscr{X}$ and assume there exist $y_1, y_2 \in B(\hilbert)$ satisfying equalities $y_1 = \Phi(x)$ and $y_2 = \Phi(x)$. Linearity of $\Phi$ and the fact, that $\Phi(0) = 0$ together imply $y_1 - y_2 = \Phi (x-x) = 0$, so $y_1 = y_2$ and hence $\Phi$ is a well-defined linear map.

\begin{lemma}\label{lemma:PhiCP}
    The map $\Phi : \mathscr{X} \to B(\hilbert)$ defined via \eqref{eq:PhiDef} is c.p..
\end{lemma}

\begin{proof}
Set $\eta : \matra{n}{\mathscr{D}}\to \bigoplus_{k=1}^{m}\matra{n}{\mathscr{A}}$ as
\begin{equation}\label{eq:EtaDef}
    \eta \left(E_{ij}\otimes\operatorname{diag}{\{a_k\}_{k=1}^{m}}\right) = \operatorname{diag}{\left\{E_{ij}\otimes a_k\right\}_{k=1}^{m}}
\end{equation}
and then extend by linearity. For $n\in\naturals$, let us take any $h_1, h_2 \in \matra{n}{\mathscr{D}}$ so that
\begin{equation}
    h_1 = \left[\operatorname{diag}{\left\{a^{(i,j)}_{k}\right\}_{k=1}^{m}}\right]_{i,j}, \quad h_2 = \left[\operatorname{diag}{\left\{b^{(i,j)}_{k}\right\}_{k=1}^{m}}\right]_{i,j}
\end{equation}
and
\begin{equation}
    h_1 h_2 = \left[\operatorname{diag}{\left\{\sum_{l=1}^{n}a^{(i,l)}_{k} b^{(l,j)}_{k}\right\}_{k=1}^{m}}\right]_{i,j}.
\end{equation}
With direct calculation we check
\begin{equation}
    \eta(h_1 h_2) = \operatorname{diag}{\left\{\left[\sum_{l=1}^{n}a^{(i,l)}_{k} b^{(l,j)}_{k}\right]_{i,j}\right\}}_{k} = \eta(h_1)\eta(h_2)
\end{equation}
and $\eta(h^\hadj) = \eta(h)^\hadj$, so $\eta$ is a *-homomorphism and an invertible positive map. This yields $h \in \matra{n}{\mathscr{X}}$, $h = \left[\operatorname{diag}{\left\{\phi_k (a_{ij})\right\}}_{k=1}^{m}\right]_{i,j}$, is positive iff
\begin{equation}
    \eta (h) = \operatorname{diag}{\left\{\left[\phi_k (a_{ij})\right]_{i,j}\right\}_{k=1}^{m}}
\end{equation}
is also positive. Positivity of the latter is equivalent to positivity of all diagonal blocks, i.e.~$\left[\phi_k (a_{ij})\right]_{i,j}$ must be positive for all $k$. For any such positive $h$ observe that
\begin{align}\label{eq:idnPhi}
    (\id{n}\otimes\Phi)(h) &= \left[\Phi(\operatorname{diag}{\left\{\phi_k (a_{ij})\right\}}_{k=1}^{m})\right]_{i,j} \\
    &= \left[\Phi\left(\left(\begin{array}{cccc}
        \phi_1(a) & 0 & \cdots & 0 \\
        0 & \phi_2(a) & \cdots & 0 \\
        \vdots & \vdots & \ddots & \vdots \\
        0 & 0 & \cdots & \phi_m(a)
    \end{array} \nonumber
\right)\right)\right]_{i,j} = \left[\varphi(a_{ij})\right]_{i,j}
\end{align}
which is also positive by the assumption. Hence, $\id{n}\otimes\Phi$ is positive for all $n\in\naturals$ i.e.~$\Phi$ is completely positive.
\end{proof}

With the above lemma at hand, we are ready to prove Theorem \ref{thm:FinDecEquivCond}.

\begin{proof}[Proof of Theorem \ref{thm:FinDecEquivCond}]
We start with sufficiency. Let $\varphi$ be $\phi$-decomposable and pick any $\left[a_{ij}\right]_{i,j}\in\Gamma_{n}^{+}(\phi)$. Then,
\begin{equation}\label{eq:GenDecProofRight}
    \left[\varphi(a_{ij})\right]_{i,j} = \left[ \sum_{k=1}^{m} \psi_k \circ \phi_k \left(a_{ij}\right) \right]_{i,j} = \sum_{k=1}^{m} \left(\id{n}\otimes\psi_k\right)\left(\left[\phi_k (a_{ij})\right]_{i,j}\right),
\end{equation}
which is positive, since $\left[\phi_k (a_{ij})\right]_{i,j} \in \matra{n}{\mathscr{A}}^+$ by the assumption and $\id{n}\otimes\psi_k$ is positive on $\matra{n}{\mathscr{A}}$ by complete positivity of $\psi_k$.

To show necessity we employ the construction tailored earlier, with a C*-algebra $\mathscr{W}$ specified by \eqref{eq:WdefinitionCstar} and maps $\xi$, $\Phi$ given by \eqref{eq:xiDef} and \eqref{eq:PhiDef}, respectively. Map $\Phi$ was shown to be c.p.~on the operator system $\mathscr{X} = \xi(\mathscr{A})$ in Lemma \ref{lemma:PhiCP} and as such admits a c.p.~extension $\tilde{\Phi} : \mathscr{W}\to B(\hilbert)$ by virtue of the Arveson's extension theorem. Then, the Stinespring's dilation theorem \cite{Stinespring_1955} yields existence of a Hilbert space $\mathcal{K}$, a bounded linear operator $V : \hilbert\to\mathcal{K}$ and a *-homomorphism $\pi : \mathscr{W} \to B(\mathcal{K})$ such that $\tilde{\Phi}(x) = V^\hadj \pi (x) V$. Set $\psi_k : \mathscr{A}\to B(\hilbert)$ and $\pi_k : \mathscr{A}\to B(\mathcal{K})$ by
\begin{equation}\label{eq:piPsidef}
    \pi_k(a) = \pi (E_{kk}\otimes a), \quad \psi_k (a) = V^\hadj \pi_k (a) V.
\end{equation}
It is straightforward to confirm $\pi_k$ are *-homomorphisms; thus $\psi_k$ are completely positive. Immediately,
\begin{align}\label{eq:varphiSOTseries}
    \sum_{k=1}^{m} \psi_k \circ \phi_k(a) =  \tilde{\Phi}\left( \sum_{k=1}^{m}E_{kk}\otimes\phi_k(a)\right) = \varphi (a)
\end{align}
for every $a\in\mathscr{A}$. The proof is concluded.
\end{proof}

\begin{remark}
    Assumed unitality of maps $\phi_k$ may seem as a severe limitation for the robustness of the above result, however is necessary for the subspace $\mathscr{X} = \xi(\mathscr{A})$ to be an operator system so that the Arveson's extension theorem may be applied. If unitality is dropped, then even after reverting to the unitization the construction may fail unless $\mathscr{X}$ is a C*-subalgebra in $\matra{n}{\mathscr{A}}$. A case similar to this, albeit in an infinite dimensional setting, is covered in a larger extent in Sections \ref{sec:InfDM} and \ref{sec:GDMnonunital}.
\end{remark}

\subsubsection{Properties}

Recall that the C*-algebra $B(\mathcal{\hilbert})$ may be endowed with a weak operator topology induced by a family of seminorms $T\mapsto \left| \iprod{x}{Ty}\right|$, $x,y \in \mathcal{\hilbert}$. The topology of pointwise convergence on $B(\mathscr{A},B(\mathcal{\hilbert}))$ so induced is called the \emph{BW topology}. Accordingly, we say a bounded net of maps $\left(\varphi_\lambda\right)_\lambda$ converges to a map $\varphi$ in BW topology if and only if $\varphi_\lambda (a) \to \varphi(a)$ in weak operator topology in $B(\mathcal{\hilbert})$, for every $a\in\mathscr{A}$. It is a known result that cones of c.p., c.cp.~and decomposable maps are all BW closed \cite{Stoermer2013}. This property then extends to the finitely decomposable case, as we show below.

\begin{theorem}\label{thm:DphiFinClosed}
    The set $\gdec{\phi}{\mathscr{A}}{\hilbert}$ is closed in the BW topology in $B(\mathscr{A},B(\hilbert))$.
\end{theorem}

\begin{proof}
    Let $\left(\varphi_\alpha\right)_\alpha$ be a bounded net in $\gdec{\phi}{\mathscr{A}}{\hilbert}$ converging to $\varphi$ pointwise in weak operator topology. Let $\{e_i\}$ be a standard basis spanning $\complexes^n$ for some $n\in\naturals$. Take $[a_{ij}]_{i,j} \in \Gamma_{n}^{+}(\phi)$ and $x,y \in \hilbert^n$ such that $x = \sum_{i=1}^{n} e_i \otimes x_i$ and $y = \sum_{i=1}^{n} e_i \otimes y_i$ for some vectors $x_i, y_i \in \hilbert$. Convergence of $\left(\varphi_\alpha\right)_\alpha$ implies existence of $\alpha_{ij}$, for $i,j \in \{1,...,n\}$, such that
    \begin{equation}
        \alpha > \alpha_{ij} \Rightarrow \left|\iprod{x_i}{\left(\varphi_\alpha (a) - \varphi(a)\right)y_j}\right| < \frac{\epsilon}{n^2},
    \end{equation}
    for any $a\in\mathscr{A}$, $\epsilon > 0$. Then, for $\alpha > \max_{i,j}{\alpha_{ij}}$ we have
    \begin{align}
        \left|\iprod{x}{\left[\varphi_\alpha (a_{ij}) - \varphi (a_{ij})\right]_{i,j}y}\right| < \sum_{i,j=1}^{n} \frac{\epsilon}{n^2} = \epsilon,
    \end{align}
    so $\left[\varphi_\alpha (a_{ij})\right]_{i,j}\to\left[\varphi (a_{ij})\right]_{i,j}$ in weak operator topology in $\matra{n}{B(\hilbert)}$. Since matrices $\left[\varphi_\alpha (a_{ij})\right]_{i,j}$ are positive for all $\alpha$ by Theorem \ref{thm:FinDecEquivCond} and $\matra{n}{B(\hilbert)}^+$ is weak operator closed, necessarily $\left[\varphi (a_{ij})\right]_{i,j}$ is also positive, for all $n$. Therefore the condition \eqref{eq:OnlyIfAssumInfinite} holds (with $\phi \in c_0 (B(\mathscr{A}))$) and hence Theorem \ref{thm:FinDecEquivCond} yields $\varphi$ must also be $\phi$-decomposable, $\varphi \in \gdec{\phi}{\mathscr{A}}{\hilbert}$. The set is then closed, as claimed.
\end{proof}

Recall that a proper convex cone $C$ in $B(\mathscr{A},B(\hilbert))$ is called a \emph{mapping cone} if it is closed in BW topology and is invariant with respect to compositions with c.p.~maps, in the sense that for any $\varphi \in C$ and any  $\psi_1 \in \cpe{B(\hilbert)}$, $\psi_2 \in \cpe{\mathscr{A}}$ we have $\psi_1 \circ \varphi \circ \psi_2 \in C$. If only $\psi\circ\varphi$ (resp.~$\varphi\circ\psi$) is in $C$ for all c.p.~maps $\psi$ then we can call $C$ a \emph{left} (resp.~\emph{right}) \emph{mapping cone}. A simple corollary then follows:

\begin{theorem}\label{thm:DphiFinMapCone}
    The set $\gdec{\phi}{\mathscr{A}}{\hilbert}$ is a left mapping cone in $B(\mathscr{A},B(\hilbert))$.
\end{theorem}

\begin{proof}
    Convex cone structure follows from Theorem \ref{thm:DphiConvCone}. Invariance with respect to composition with c.p.~maps from the left is obvious. Finally BW closedness follows from Theorem \ref{thm:DphiFinClosed}.
\end{proof}

\subsection{Infinite decompositions}
\label{sec:InfDM}

Here we demonstrate that a result similar to Theorem \ref{thm:FinDecEquivCond} can also hold for truly countably decomposable maps, i.e.~expressible by \eqref{eq:GenDecDefinition} with an \emph{infinite} number of terms.

\subsubsection{Characterization}

For well-behavedness of certain sums we restrict to the case when the image $\mathscr{X} = \xi(\mathscr{A})$ is a C*-algebra on its own right, which happens when all $\phi_k$ are *-homomorphisms, for instance; moreover we assume $\phi$ is a \emph{vanishing sequence}, denoted $\phi \in c_0 (B(\mathscr{A}))$, namely when $\left(\|\phi_k\|\right)_k$ is uniformly bounded and converges to $0$ as $k\to\infty$.

\begin{theorem}\label{thm:GenDecEquivCond}
    Let $\mathscr{A}$ be a unital C*-algebra and let $\left(\phi_k\right)_k \in c_0(B(\mathscr{A}))$ be a sequence of *-maps on $\mathscr{A}$. Denote $\xi(a) = \operatorname{diag}{\left\{\phi_k (a)\right\}_{k}}$ and $\mathscr{X} = \xi(\mathscr{A})$. If $\mathscr{X}$ happens to be a C*-algebra then the map $\varphi \in B(\mathscr{A},B(\hilbert))$ is $\phi$-decomposable if and only if 
    \begin{equation}\label{eq:OnlyIfAssumInfinite}
        \forall \, n\in\naturals : \left[a_{ij}\right]_{i,j} \in \Gamma_{n}^{+}(\phi) \Rightarrow \left[\varphi(a_{ij})\right]_{i,j} \in \matra{n}{B(\hilbert)}^+ .
    \end{equation}
\end{theorem}

\begin{proof}
    Showing the claim requires a moderate modification of the proof of Theorem \ref{thm:FinDecEquivCond}. For sufficiency, pick any $n\in\naturals$ and any $\left[a_{ij}\right] \in \Gamma_{n}^{+}(\phi)$; since $\sum_{k=1}^{\infty}\psi_k \circ \phi_k (a)$ converges in norm for any $a\in\mathscr{A}$, one easily shows a sequence of partial sums
    \begin{equation}\label{eq:partialSumVarphi}
        \left[\sum_{k=1}^{n}\psi_k\circ\phi_k(a_{ij})\right]_{i,j}
    \end{equation}
    converges to $\left[\varphi(a_{ij})\right]_{i,j}$ in weak operator topology in $\matra{n}{B(\hilbert)}$ as $n\to\infty$. Maps $\psi_k$ are c.p.~and $\left[\phi_k(a_{ij})\right]_{i,j}$ are all positive by the assumption; hence all matrices \eqref{eq:partialSumVarphi} are also positive. Since $\matra{n}{B(\hilbert)}^+$ is weak operator closed, $\left[\varphi(a_{ij})\right]_{i,j}$ must be positive as well.

    Now we show necessity. Let $\left\{e_k\right\}_k$ be a standard basis spanning $l^2$ and denote $E_{kl} = \iprod{e_l}{\cdot}e_k \in B(l^2)$. Take any $w$ from $B(l^2)\odot\mathscr{A}$, the algebraic tensor product, and define its spatial norm
    \begin{equation}
        \|w\|_* = \| \left(\id{}\otimes\pi_\mathscr{A}\right)(w) \| ,
    \end{equation}
    where $\pi_\mathscr{A} : \mathscr{A}\to B(H_\mathscr{A})$ is (any) faithful *-representation of $\mathscr{A}$ on an external Hilbert space $H_\mathscr{A}$. We define $\mathscr{W}$ as the \emph{spatial tensor product}, i.e.~a completion
    \begin{equation}\label{eq:WdefinitionSpatial}
        \mathscr{W} = \overline{B(l^2)\odot\mathscr{A}}^{\|\cdot\|_*} = B(l^2) \hat{\otimes}_\mathrm{min} \mathscr{A}
    \end{equation}
    of $B(l^2)\odot\mathscr{A}$, which is also a C*-algebra with a unit $\unit_\mathscr{W} = \id{}\otimes\unit_\mathscr{A}$. Let $c_0 (\mathscr{A})$ be a linear space of all vanishing sequences with elements in $\mathscr{A}$, $\left(a_k\right)_k \in c_0 (\mathscr{A})$ iff $\|a_k\|\to 0$, and consider $w$ given formally as
    \begin{equation}
        w = \sum_{k=1}^{\infty} E_{kk} \otimes a_k, \quad \left(a_k\right)_k \in c_0 (\mathscr{A}).
    \end{equation}
    Then, one shows $w$ is an element of $\mathscr{W}$ and
    \begin{equation}\label{eq:spatialNormDiagonal}
        \|w\|_* = \sup_{k\in\naturals}\|a_k\|.
    \end{equation} 
    Let $\mathscr{D}$ be a set of all such elements, for any $\left(a_k\right)_k \in c_0 (\mathscr{A})$; then $\mathscr{D}$ is a norm closed linear subspace in $\mathscr{W}$, isometrically isomorphic to $c_0 (\mathscr{A})$ by \eqref{eq:spatialNormDiagonal}. Now, let $\phi \in c_0 (B(\mathscr{A}))$ and define $\xi(a) = \operatorname{diag}{\left\{\phi_k (a)\right\}_{k}}$. Then we see $\xi(a) \in c_0(\mathscr{A})$ for any $a\in\mathscr{A}$ and so $\mathscr{X} =\xi(\mathscr{A})$ is a *-invariant linear subspace in $\mathscr{D}$. That said,
    \begin{equation}\label{eq:xiInfinite}
        \xi(a) = \sum_{k=1}^{\infty}E_{kk}\otimes\phi_k (a)
    \end{equation}
    converges in spatial norm \eqref{eq:spatialNormDiagonal} in $\mathscr{D}$. However, contrary to the finitely decomposable case, $\mathscr{X}$ is not an operator system in $\mathscr{W}$: indeed, since $\phi$ is assumed to vanish in norm, maps $\phi_k$ are not allowed to be all unital; therefore $\mathscr{X}$ does not contain $\unit_\mathscr{W}$ by construction (we remark here that loosening the condition $\phi\in c_0(B(\mathscr{A}))$ in order to allow unitality spoils the norm convergence of \eqref{eq:xiInfinite}). Define a bounded map $\Phi : \mathscr{X}\to B(\hilbert)$ by setting $\varphi (a) = \Phi\circ\xi(a)$ as in Theorem \ref{thm:FinDecEquivCond}. We can adjust the proof of Lemma \ref{lemma:PhiCP} to cover infinite dimensional sequence $\phi$ by setting a linear map $\eta : \matra{n}{\mathscr{D}}\to\bigoplus_{k\in\naturals}\matra{n}{\mathscr{A}}$ as
    \begin{equation}
        \eta (E_{ij}\otimes \operatorname{diag}{\left\{a_k\right\}_k}) = \operatorname{diag}{\left\{E_{ij}\otimes a_k\right\}_k},
    \end{equation}
    for $E_{ij}$ being matrix units in $\matr{n}$ and $\left(a_k\right)_k \in c_0 (\mathscr{A})$, and showing it is a *-homomorphism, in virtually the same way. Then, demonstrating $\id{n}\otimes\Phi$ is a positive map is immediate from the assumption $\left[\varphi(a_{ij})\right]_{i,j} \in \matra{n}{B(\hilbert)}^+$, for all $n$, i.e.~$\Phi$ is completely positive.
    
    To overcome the difficulty of $\mathscr{X}$ not being an operator system, we pass to the unitization. In the case when $\mathscr{X}$ is a (nonunital) C*-algebra, as we assumed at the beginning, map $\Phi$ is automatically completely bounded on $\mathscr{X}$. By virtue of Lemma \ref{lemma:Phi0ucp} (see the Appendix) it then extends to a c.p.~map $\Phi^\sim : \mathscr{X}\oplus \complexes 1 \to B(\hilbert)$ set by
    \begin{equation}
        \Phi^\sim ((x,z)) = \Phi(x) + z \cbnorm{\Phi}I
    \end{equation}
    for $(x,z) \in \mathscr{X}\oplus \complexes 1$, $\cbnorm{\Phi}  = \sup_{n\in\naturals}\|\id{n}\otimes \Phi\|$ being the CB norm of $\Phi$ and $I$ being the identity operator in $B(\hilbert)$. As $\mathscr{X}\oplus \complexes 1$ is an operator system in $\mathscr{W}^\sim = \mathscr{W}\oplus \complexes 1$, the forced unitization of $\mathscr{W}$, the map $\Phi^\sim$ admits a bounded c.p.~extension $\Psi^\sim : \mathscr{W}^\sim \to B(\hilbert)$ by the Arveson's extension theorem. Define $\tilde{\Phi} : \mathscr{W} \to B(\hilbert)$ by setting $\tilde{\Phi} = \left.\Psi^\sim\right|_\mathscr{W}((\cdot,0))$. Then, for any $x\in\mathscr{X}$ we have
    \begin{equation}
        \tilde{\Phi}(x) = \left.\Psi^\sim\right|_\mathscr{W}((x,0)) = \Phi^\sim ((x,0)) = \Phi(x),
    \end{equation}
    i.e.~$\tilde{\Phi}$ is a bounded c.p.~extension of $\Phi$ on $\mathscr{W}$. The remaining reasoning is similar to the finitely decomposable case: by Stinespring's theorem, $\tilde{\Phi}$ may be expressed as $V^\hadj \pi(\cdot) V$ for $V : \hilbert\to \mathcal{K}$ bounded, $\pi : \mathscr{W}\to B(\mathcal{K})$ a *-homomorphism and $\mathcal{K}$ an external Hilbert space. To finalize the proof define $\pi_k$ and $\psi_k$ as in \eqref{eq:piPsidef}; then
    \begin{equation}
        \sum_{k=1}^{\infty} \psi_k \circ \phi_k(a) =  \tilde{\Phi}\left( \sum_{k=1}^{\infty}E_{kk}\otimes\phi_k(a)\right) = \varphi (a)
    \end{equation}
    where the left-most series converges in norm, for every $a\in\mathscr{A}$, by boundedness of $\tilde{\Phi}$ and norm convergence of \eqref{eq:xiInfinite}.
\end{proof}

\subsubsection{Properties}

\begin{theorem}\label{thm:DphiContNormClosed}
    The set $\gdec{\phi}{\mathscr{A}}{\hilbert}$ is closed in the BW topology in $B(\mathscr{A},B(\hilbert))$ under assumptions of Theorem \ref{thm:GenDecEquivCond}.
\end{theorem}

\begin{proof}
It suffices to reapply the reasoning from Theorem \ref{thm:DphiFinClosed} for weak operator convergence and then employ Theorem \ref{thm:GenDecEquivCond} to show the closedness.
\end{proof}

\begin{theorem}\label{thm:DphiCountMapCone}
    The set $\gdec{\phi}{\mathscr{A}}{\hilbert}$ is a left mapping cone in $B(\mathscr{A},B(\hilbert))$.
\end{theorem}

\begin{proof}
    Let $\psi \in \cpe{B(\hilbert)}$. If $\sigma_n = \sum_{k=1}^{n}\psi_k \circ \phi_k$ converges pointwise in norm to $\varphi$, then $\left(\psi\circ\sigma_n\right)_n$ converges to $\psi\circ\varphi$ since $\psi$ is bounded and norm continuous; thus $\gdec{\phi}{\mathscr{A}}{\hilbert}$ is invariant with respect to compositions with c.p.~maps from the left. Theorems \ref{thm:DphiConvCone} and \ref{thm:DphiContNormClosed} then yield the claim.
\end{proof}

\subsection{Nonunital C*-algebra}
\label{sec:GDMnonunital}

It may be of interest to investigate the case when $\mathscr{A}$ has no unit; it turns out that the situation in this case is comparative. Below we present a result which is similar in essence to Theorem \ref{thm:GenDecEquivCond} of unital case, with the same restrictive assumption imposed on $\mathscr{X} = \xi(\mathscr{A})$, which here is also a C*-algebra. The scope of this result covers both countably and finitely decomposable cases.

\begin{theorem}\label{thm:GenDecEquivCondNonunital}
    Let $\mathscr{A}$ be a nonunital C*-algebra and let $\left(\phi_k\right)_{k} \in c_0 (B(\mathscr{A}))$ be a vanishing sequence of bounded *-maps on $\mathscr{A}$ (possibly of finite length). Denote $\xi(a) = \operatorname{diag}{\left\{ \phi_k (a) \right\}_{k}}$ and $\mathscr{X} = \xi(\mathscr{A})$. If $\mathscr{X}$ happens to be a C*-algebra then the claim of Theorem \ref{thm:GenDecEquivCond} applies.
\end{theorem}

\begin{proof}
    Sufficiency is shown in identical way, since it does not rely on unitality of $\mathscr{A}$. For necessity, the construction is almost the same with $\mathscr{W}$ and $\xi$ defined just like therein by \eqref{eq:WdefinitionSpatial} and \eqref{eq:xiInfinite}, respectively. Map $\Phi$, still given via equality $\varphi = \Phi\circ\xi$, remains c.p.~since unitality plays no role in the construction outlined in Lemma \ref{lemma:PhiCP}. The remaining reasoning then follows the same unitization trick with $\mathscr{X}\oplus\complexes 1$ being a C*-subalgebra, and an operator system, in $\mathscr{W}^\sim = \mathscr{W}\oplus \complexes 1$. Map $\Phi$ again extends, with the help of Lemma \ref{lemma:Phi0ucp}, to a c.p.~map $\Phi^\sim : \mathscr{X}^\sim \oplus \complexes 1 \to B(\hilbert)$ defined by
    \begin{equation}
        \Phi^{\sim} ((x,z)) = \Phi_0 (x) + z \cbnorm{\Phi} I,
    \end{equation}
    where $I$ is the identity operator in $B(\hilbert)$ and $\cbnorm{\Phi}$ is a CB norm of $\Phi$. This in turn extends to a c.p.~map $\Psi^{\sim} : \mathscr{W}^\sim \to B(\hilbert)$ of the same norm by the Arveson's extension theorem. Again, $\tilde{\Phi} = \Psi^{\sim}((\cdot,0))$ extends $\Phi$. The remaining construction is like in the proof of Theorem \ref{thm:GenDecEquivCond}.
\end{proof}

As the final result of the article, we note the same claims of BW closedness and convex cone structure of the set $\gdec{\phi}{\mathscr{A}}{\hilbert}$, i.e.~Theorems \ref{thm:DphiContNormClosed} and \ref{thm:DphiCountMapCone}, apply in the nonunital case under assumptions of Theorem \ref{thm:GenDecEquivCondNonunital}.

\section{Acknowledgements}

The author acknowledges contribution from the GPT-5.2 by OpenAI and Gemini 2.5 Pro by Google in clarifying some convergence related issues in the proof of Theorem \ref{thm:GenDecEquivCond}, suggesting the strategy for showing Lemma \ref{lemma:Phi0ucp}, and performing a thorough proofreading of the early version of the manuscript.

\section{Summary}

We presented a notion of the countable decomposability of linear maps on C*-algebras and provided their general characterization in some chosen cases. Our findings may be regarded as an extension of results of \cite{Stormer_1982} towards more general decompositions with arbitrary number of terms. What is worth emphasizing, all the core theorems we provided remain consistent with the general spirit of the original Størmer's characterization of decomposability. Possible further research directions and open problems include, but are by no means limited to, the following:
\begin{enumerate}
    \item Extending the construction by considering weaker topologies of convergence of the decomposition \eqref{eq:GenDecDefinition}.
    \item Exploring some nontrivial choices of $\phi$ and verifying properties of cones $\gdec{\phi}{\mathscr{A}}{\hilbert}$ constructed in this manner.
    \item Replacing the complete positivity requirement for maps $\psi_k$ by $n$-positivity. This modification would lead to entirely different characterization theorems and result in a much broader notion of \emph{countable $n$-decomposability}.
    \item Finally, establishing a stronger connection to positive maps. As the very notion of countable decomposability originates from decomposable positive maps, this direction of investigation seems the most natural one. It is worth verifying if our approach is able to give any more insight into the general structure of positive maps itself, at least in a matrix algebra case.
\end{enumerate}

\appendix

\section{}

\begin{lemma}\label{lemma:Phi0ucp}
    Let $\mathscr{A}$ be a nonunital C*-algebra, $\mathscr{M} \subset \mathscr{A}$ a C*-subalgebra, $H$ a Hilbert space and $T : \mathscr{M}\to B(H)$ a c.p.~map. Then, its extension $T^\sim : \mathscr{M} \oplus \complexes 1 \to B(H)$ given via
    \begin{equation}
        T^\sim ((a,z)) = T(a) + z \cbnorm{T} I,
    \end{equation}
    for $I$ the identity operator on $H$ and $\cbnorm{T}$ the CB norm of $T$, is completely positive.
\end{lemma}

\begin{proof}
    Denote $\mathscr{M}^\sim = \mathscr{M}\oplus\complexes 1$. We must show $T_{n}^{\sim} = \id{n}\otimes T^\sim$ is positive on $\matra{n}{\mathscr{M}^\sim}$, for all $n\in\naturals$. Let matrices $\left[x_{ij}\right]_{i,j} \in \matra{n}{\mathscr{M}}$ and $\Lambda \in \matr{n}$ be arbitrary. Then, any element $A \in \matra{n}{\mathscr{M}^\sim}$ admits a matrix representation
    \begin{equation}
        A = X + \Lambda \otimes \unit_\sim ,
    \end{equation}
    where $X = \left[(x_{ij},0)\right]_{i,j}$ and $\unit_\sim = (0,1)$ is a unit in $\mathscr{A}^\sim = \mathscr{A}\oplus \complexes 1$, the forced unitization of $\mathscr{A}$. Assume $A$ is a positive element in $\matra{n}{\mathscr{M}^\sim}$. Let $q : \mathscr{A}^\sim \to \complexes$ be a quotient map, $q ((a,z)) = z$. One checks $q$ is a *-homomorphism and a c.p.~map in consequence; hence $q_n = \id{n}\otimes q$ is positive. Since $q_n$ acts by restricting $A$ to its scalar part, $q_n (A) = \Lambda$, it must hold that $\Lambda$ is positive semidefinite. Define $\Lambda_\epsilon = \Lambda + \epsilon I_n$ and $A_\epsilon = X + \Lambda_\epsilon \otimes \unit_\sim$ for $\epsilon > 0$ and $I_n$ an identity matrix in $\matr{n}$. Since $A_\epsilon - A = \epsilon I_n \otimes \unit_\sim$ we see $A_\epsilon > A$ and so $\left(A_\epsilon\right)_\epsilon$ is a net of strictly positive elements, approximating $A$ in norm as $\epsilon \to 0$. Element $A_\epsilon$ is positive if and only if
    \begin{equation}
        (\Lambda_{\epsilon}^{-\frac{1}{2}}\otimes \unit_\sim)A_\epsilon (\Lambda_{\epsilon}^{-\frac{1}{2}}\otimes \unit_\sim) = Y_\epsilon + I_n \otimes \unit_\sim
    \end{equation}
    is positive as well, where we introduced
    \begin{equation}
        Y_\epsilon = (\Lambda_{\epsilon}^{-\frac{1}{2}}\otimes \unit_\sim)X (\Lambda_{\epsilon}^{-\frac{1}{2}}\otimes \unit_\sim).
    \end{equation}
    Positivity of $Y_\epsilon + I_n \otimes \unit_\sim$ is equivalent to $Y_\epsilon$ being self-adjoint and the spectrum $\spec{Y_\epsilon + I_n \otimes \unit_\sim}$ being nonnegative, i.e.
    \begin{equation}
        \spec{Y_\epsilon} \subset [-1,\infty ).
    \end{equation}
    Denote with $\spec{Y_\epsilon}_\pm$ the positive (nonnegative) and negative part of the spectrum, respectively. As $\mathscr{M}$ is a C*-algebra, it admits a well-defined functional calculus; the same is then true for $\mathscr{M}^\sim$ and $\matra{n}{\mathscr{M}^\sim}$. Let then $E$ be the spectral measure of $Y_\epsilon$. Define
    \begin{equation}
        Y_{\epsilon}^{\pm} = \int\limits_{\spec{Y_\epsilon}_\pm} |\lambda| \, dE(\lambda),
    \end{equation}
    so that $Y_{\epsilon}^{\pm}$ are positive, satisfy $Y_{\epsilon}^{+} Y_{\epsilon}^{-} = 0$ and $Y_\epsilon = Y_{\epsilon}^{+} - Y_{\epsilon}^{-}$. Since $Y_{\epsilon}^{-}$ is self-adjoint and its spectrum $\spec{Y_\epsilon}_-$ is contained in $[-1,0]$ we see $\| Y_\epsilon \| \leqslant 1$. It is an easy exercise to show that due to its definition, $Y_{\epsilon}^{\pm}$ have a structure $\left[(y_{ij}^{\pm}(\epsilon),0)\right]_{i,j}$ for some $\hat{y}_{\epsilon}^{\pm} = \left[y_{ij}^{\pm}(\epsilon)\right]_{i,j}\in \matra{n}{\mathscr{M}}$, so are contained in a subspace of $\matra{n}{\mathscr{M}^\sim}$ canonically isometrically isomorphic to $\matra{n}{\mathscr{M}}$; in particular $\left\|Y_{\epsilon}^{\pm}\right\| = \left\| \hat{y}_{\epsilon}^{\pm} \right\|$. Acting with $T_{n}^{\sim}$ we obtain
    \begin{equation}\label{eq:TnSimOnYplusUnit}
        T_{n}^{\sim}(Y_\epsilon + I_n \otimes \unit_\sim) = T_{n}^{\sim}(Y_{\epsilon}^{+}) - T_{n}^{\sim}(Y_{\epsilon}^{-}) + \cbnorm{T} I_n \otimes I,
    \end{equation}
    since $T_{n}^{\sim}(I_n \otimes \unit_\sim) = \left[ T^\sim ((0,1)) \right]_{i,j} = \cbnorm{T}\left[\delta_{ij}I\right]_{i,j}$ by definition of $T^\sim$. Notice
    \begin{equation}
        T_{n}^{\sim}(Y_{\epsilon}^{-}) = \left[ T^\sim ((y_{ij}^{-}(\epsilon),0)) \right]_{i,j} = \left[T(y_{ij}^{-}(\epsilon))\right]_{i,j} = T_n (\hat{y}^{-}_\epsilon)
    \end{equation}
    for $T_n = \id{n} \otimes T : \matra{n}{\mathscr{M}}\to \matra{n}{B(H)}$. Since we assumed $T$ was completely bounded and $\left\| Y_{\epsilon}^{-}\right\| \leqslant 1$, for any $n\in\naturals$ we have en estimation
    \begin{equation}
        \left\|T_{n}^{\sim}(Y_{\epsilon}^{-}) \right\| \leqslant \cbnorm{T} \left\|\hat{y}^{-}_{\epsilon}\right\| = \cbnorm{T} \left\| Y_{\epsilon}^{-} \right\| \leqslant \cbnorm{T}.
    \end{equation}
    Both elements $T_{n}^{\sim}(Y_{\epsilon}^{+})$, $T_{n}^{\sim}(Y_{\epsilon}^{-})$ are positive by the assumed complete positivity of $T$. In particular,
    \begin{equation}
        \spec{T_{n}^{\sim}(Y_{\epsilon}^{-})} \subset [0, \cbnorm{T}]
    \end{equation}
    which yields, with the use of the spectral theorem, that the spectrum
    \begin{equation}
        \spec{\cbnorm{T}I_n \otimes I - T_{n}^{\sim}(Y_{\epsilon}^{-})} = \{\cbnorm{T} - \lambda : \lambda \in [0, \cbnorm{T}] \}
    \end{equation}
    is contained in $[0,\cbnorm{T}]$, i.e.~$\cbnorm{T}I_n \otimes I - T_{n}^{\sim}(Y_{\epsilon}^{-})$ is positive. In consequence, the right hand side of \eqref{eq:TnSimOnYplusUnit} is positive for all $n$ and all $\epsilon > 0$. Denote $\mathcal{Y}_\epsilon = T_{n}^{\sim}(Y_\epsilon + I_n \otimes \unit_\sim)$. One checks with a direct computation that any map $\phi_n = \id{n}\otimes\phi$, where $\phi$ acts on $\mathscr{M}^\sim$, satisfies the intertwining relation
    \begin{equation}
        \phi_n \left((K \otimes \unit_\sim)X(L \otimes \unit_\sim)\right) = (K \otimes \unit_\sim) \phi_n(X) (L \otimes \unit_\sim)
    \end{equation}
    for any $K,L \in \matr{n}$ and $X \in \matra{n}{\mathscr{M}^\sim}$; this yields
    \begin{equation}
        T_{n}^{\sim} (A_\epsilon) = \left(\Lambda_{\epsilon}^{\frac{1}{2}}\otimes\unit_\sim\right) \mathcal{Y}_\epsilon \left(\Lambda_{\epsilon}^{\frac{1}{2}}\otimes\unit_\sim\right)
    \end{equation}
    defines a net of positive elements in $\matra{n}{B(H)}$. Net $\left(A_\epsilon\right)_\epsilon$ was norm convergent to $A$, so we have $\| A_\epsilon - A\| < \frac{\delta}{\|T^{\sim}_{n}\|}$ for some $\delta > 0$ and $\epsilon$ small enough; hence $\| T_{n}^{\sim} (A_\epsilon - A) \| < \delta$ and $\left(T_{n}^{\sim} (A_\epsilon)\right)_\epsilon$ converges to $T_{n}^{\sim}(A)$ in norm as $\epsilon \to 0$. Since $\matra{n}{B(H)}^+$ is norm closed, $T_{n}^{\sim}(A)$ must be positive, for any $A \in \matra{n}{\mathscr{M}^\sim}^+$ and any $n$. This shows $T^\sim$ is completely positive.
\end{proof}

\bibliographystyle{unsrt}
\bibliography{GeneralizedDecomposableMaps}

\end{document}